\documentclass[11pt]{amsart}

\makeindex
\usepackage[all, cmtip]{xy}
\usepackage{amssymb}
\usepackage{amsmath}
\usepackage{amsthm}
\usepackage{amscd}
\usepackage{amsfonts}
\usepackage{amssymb}
\usepackage{multirow}
\usepackage{epic}
\usepackage{graphicx}
\usepackage{pgf,tikz}
\usepackage{mathrsfs}
\usetikzlibrary{arrows}
\usepackage{enumitem}
\usepackage[normalem]{ulem} 

\definecolor{qqqqff}{rgb}{0,0,1}
\definecolor{xdxdff}{rgb}{0.49,0.49,1}
\definecolor{uuuuuu}{rgb}{0.27,0.27,0.27}
\definecolor{cqcqcq}{rgb}{0.75,0.75,0.75}

\usetikzlibrary{arrows}

\newtheorem{theorem}{Theorem}[section]
\newtheorem{lemma}[theorem]{Lemma}
\newtheorem{proposition}[theorem]{Proposition}

\theoremstyle{corollary}
\newtheorem{corollary}[theorem]{Corollary}

\theoremstyle{definition} 
\newtheorem{definition}[theorem]{Definition}
\newtheorem{definition-lemma}[theorem]{Definition-Lemma}

\theoremstyle{remark}
\newtheorem{remark}[theorem]{Remark}

\numberwithin{equation}{section}

\newcommand{\R}{\mathbb{R}}
\newcommand{\Z}{\mathbb{Z}}

\newcommand{\Q}{\mathbb{Q}}

\newcommand{\bm}{\mathbf B_-}
\newcommand{\bp}{\mathbf B_+}

\def\Supp{\operatorname{Supp}}

\def\Int{\operatorname{Int}}

\def\Supp{\operatorname{Supp}}

\def\bNM{\operatorname{bNM}}
\def\NM{\operatorname{NM}}
\def\NE{\operatorname{NE}}
\def\Nef{\operatorname{Nef}}

\def\Amp{\operatorname{Amp}}
\def\Bigdiv{\operatorname{Big}}

\newcommand\eps{\varepsilon}

\title[On a generalized Batyrev's cone conjecture]{On a generalized Batyrev's cone conjecture}

\begin{document}

\author{Sung Rak Choi}
\address{Department of mathematics, Yonsei University, 50 Yonsei-Ro, Seodaemun-Gu, Seoul 03722, Korea}
\email{sungrakc@yonsei.ac.kr}

\author{Yoshinori Gongyo}
\address{Graduate School of Mathematical Sciences, the University of Tokyo,
3-8-1 Komaba, Meguro-ku, Tokyo 153-8914, Japan}
\email{gongyo@ms.u-tokyo.ac.jp}



\begin{abstract}
We discuss some variants of cone theorem for movable curves in any codimensions.
\end{abstract}


\maketitle

\section{Introduction}

One of the fundamental results behind the development of the minimal model program (MMP) is the following celebrated theorem.

\begin{theorem}[{Cone Theorem, \cite[Theorem 3.5]{KM}}]\label{thrm-Cone theorem}
Let $X$ be a $\Q$-factorial normal projective variety and $(X,\Delta)$ be a klt pair.
Then there exists a countable set $I$ of $(K_X+\Delta)$-negative extremal rays of $\overline{\NE}(X)$ such that
$$
\overline{\NE}(X)=\overline{\NE}(X)_{K_X+\Delta\geq 0}+\sum_{R_i\in I}R_i
$$
and for any ample divisor $H$, there exists a finite subset $J_H\subseteq I$ such that
$$
\overline{\NE}(X)=\overline{\NE}(X)_{K_X+\Delta+H\geq 0}+\sum_{R_i\in J_H}R_i.
$$
For any ample divisor $H$, each extremal ray $R_i\in J_H$ is spanned by a rational curve $C_i$ on $X$, i.e., $R_i=\mathbb R_{\geq0}[C_i]$.
\end{theorem}

This describes the structure of the Mori cone $\overline\NE(X)$.
Roughly speaking, by contracting each $(K_X+\Delta)$-negative extremal ray $R_i\in I$, we proceed to find another birationally equivalent variety $X'$ which is closer to a minimal model of $(X,\Delta)$. In this sense, the Cone Theorem basically tells us how to run the minimal model program.
In a similar manner, although partially, the following gives a description of the movable cone $\overline{\NM}(X)$, the closure of the cone spanned by movable curves on $X$.
We will use the notation
$$\overline{\NM}(X,\Delta):=\overline{\NE}(X)_{K_X+\Delta\geq 0}+\overline{\NM}(X).$$

\begin{theorem}[{cf.\cite{Batyrev},\cite{Araujo},\cite{Lehmann-cone}}]\label{thrm-batyrev}
Let $X$ be a $\Q$-factorial normal projective variety and $(X,\Delta)$ be a klt pair.
Then there exists a countable set $I'$ of $(K_X+\Delta)$-negative extremal rays $R_i$ of $\overline{\NM}(X,\Delta)$ such that
$$
\overline{\NM}(X,\Delta)=\overline{\NE}(X)_{K_X+\Delta\geq 0}+\sum_{R_i\in I'}R_i
$$
and for any ample divisor $H$, there exists a finite subset $J'_H\subseteq I'$ such that
$$
\overline{\NM}(X,\Delta+H)=\overline{\NE}(X)_{K_X+\Delta+H\geq 0}+\sum_{R_i\in J'_H}R_i.
$$
For any ample divisor $H$, each extremal ray $R_i\in J'_H$ is spanned by a movable curve $C_i$ on $X$, i.e., $R_i=\mathbb R_{\geq0}[C_i]$.
\end{theorem}

Note that the curves $C_i$ in Theorem \ref{thrm-batyrev} are not guaranteed to be rational curves as is the case in Theorem \ref{thrm-Cone theorem}. Nonetheless, each ray $R_i\in J'_H$ is associated to an extremal ray of the Mori cone $\overline{\NE}(X')$ of some birational model $X'$ of $X$. By the Cone Theorem, such extremal ray is spanned by a rational curve on $X'$ and it gives rise to a Mori fibre space.

Our goal in this paper is to treat the above two results in a fixed framework and obtain a common generalization for both results simultaneously.

Let $X$ be a $\Q$-factorial normal projective variety of dimension $d$ and $l$ be an integer such that $0\leq l\leq d-1$. A curve $C$ on $X$ is said to be \emph{movable in codimension $l$} if it belongs to a family of curves moving in a subvariety of codimension $\leq l$ in $X$.
We denote by $\overline{\NM}^l(X)$ the closure of the cone in $N_1(X)$ spanned by the classes of curves on $X$ that are movable in codimension $l$.
Note that if $\varphi:X\dashrightarrow X'$ is a small $\Q$-factorial modification, then $N^1(X)\cong N^1(X')$. Hence their dual spaces are also isomorphic: $N_1(X)\cong N_1(X')$.
Through this isomorphism, any curve $C$ on $X'$ defines a class in $N_1(X)$ although it is unclear whether such a class in $N_1(X)$ can be represented by a curve on $X$.
For a birational map $\varphi:X\dashrightarrow X'$ as above, we denote by $\overline{\NM}^l(X,X')\subseteq N_1(X)$ the image of the closed convex cone in $N_1(X')$ generated by the curves movable in some birational images of the subvarieties in $X$ of codimension $\leq l$. 
We also denote by $\overline{\bNM}^l(X)$ the closure of the cone in $N_1(X)$ spanned by all the classes in $\overline{\NM}^l(X,X')$, that is,
$$\overline{\bNM}^l(X)=\overline{\sum_{\varphi:X\dashrightarrow X'}\overline{\NM}^l(X,X')}$$
where the summation is taken over all the small $\Q$-factorial modifications $\varphi:X\dashrightarrow X'$. Note that the cone $\overline{\bNM}^{d-1}(X)$ coincides with the Mori cone $\overline{\NE}(X)$ since $\overline{\NM}^l(X,X')\subseteq\overline{\Amp}_l(X)^{\vee}$ holds for any $l$ and $\overline{\Amp}_1(X)^{\vee}=\overline{\NE}(X)$ by Kleiman.
The equality $\overline{\bNM}^0(X)=\overline{\NM}^0(X)$ follows from \cite{BDPP}.

Let $(X,\Delta)$ be a pair with some effective $\R$-divisor $\Delta$ on $X$.
The following cone is of our most interest in this paper: for an integer $l$ such that $0\leq l\leq d-1$,
$$\overline{\NE}^l(X,\Delta):=\overline{\NE}(X)_{K_X+\Delta\geq0}+\overline{\bNM}^l(X).$$
Note that the inclusion $\overline{\NE}^l(X,\Delta)\subseteq \overline{\NE}(X)$ holds in general and it is an equality when $l=d-1$.
For any divisor $H$, we denote by $\Sigma^l(X,\Delta)_{K_X+\Delta+H<0}$ the set of $(K_X+\Delta+H)$-negative extremal rays of the cone $\overline{\NE}^l(X,\Delta)$. The rays in $\Sigma^l(X,\Delta)$ are called the $l$-th coextremal rays of $(X,\Delta)$. By definition, $(d-1)$-th coextremal rays of $(X,\Delta)$ are the $(K_X+\Delta)$-negative extremal rays of $\overline{\NE}(X)$.

The following is our main result in this paper.
The cases where $l=d-1$ and $l=0$ correspond to the Cone Theorem (Theorem \ref{thrm-Cone theorem}) and Theorem \ref{thrm-batyrev}, respectively.

\begin{theorem}\label{main theorem}
Let $X$ be a $\Q$-factorial normal projective variety of dimension $d$ and $(X,\Delta)$ be a klt pair.
Then the following hold for each integer $l$ such that $0\leq l\leq d-1$:
\begin{enumerate}
\item The set $\Sigma^l(X,\Delta)$ of $l$-th coextremal rays of $(X,\Delta)$ is countable and we have
$$
\overline{\NE}^l(X,\Delta)=\overline{\NE}(X)_{K_X+\Delta\geq0}+\sum_{R_i\in \Sigma^l(X,\Delta)} R_i.
$$
\item For any ample divisor $H$, the set $\Sigma^l(X,\Delta+H)$ of $l$-th coextremal rays of $(X,\Delta+H)$ is a finite subset of $\Sigma^l(X,\Delta)$ and we have
$$
\overline{\NE}^l(X,\Delta+H)=\overline{\NE}(X)_{K_X+\Delta+H\geq0}+\sum_{R_i\in \Sigma^l(X,\Delta+H)} R_i.
$$
\item Furthermore, each ray $R_i\in \Sigma^l(X,\Delta+H)$ for any ample divisor $H$ is spanned by the numerical pull-back of a rational curve $C_i$ on some birational model $X'$ of $X$.
\end{enumerate}
\end{theorem}

\definecolor{rvwvcq}{rgb}{0.08235294117647059,0.396078431372549,0.7529411764705882}
\begin{figure}
  \centering
\begin{tikzpicture}[scale=2.2][line cap=round,line join=round,>=triangle 45,x=1cm,y=1cm]
\draw [line width=0.4pt] (1.08827687792906,2.7286739705034733)-- (1.0282962900315282,1.5136102356402312);
\draw [line width=0.4pt] (1.0282962900315282,1.5136102356402312)-- (1.5324096447430025,1.0171932680235873);
\draw [line width=0.4pt] (1.5324096447430025,1.0171932680235873)-- (2.5791182896401095,1.1480318486357262);
\draw [line width=0.4pt] (2.5791182896401095,1.1480318486357262)-- (3.0801446261687824,1.8185550647415993);
\draw [line width=0.4pt] (3.0801446261687824,1.8185550647415993)-- (3.0118070589036967,2.6517477117812978);
\draw [line width=0.4pt] (3.0118070589036967,2.6517477117812978)-- (2.6674908546065352,3.0065773879653963);
\draw [line width=0.4pt] (2.6674908546065352,3.0065773879653963)-- (1.853419780664326,3.1773023636669087);
\draw [line width=0.4pt] (1.853419780664326,3.1773023636669087)-- (1.08827687792906,2.7286739705034733);
\draw [line width=0.4pt] (0.8302616423644287,3.0065773879653963)-- (2.6984894298863025,1.0195389807688446);



    \node (v1) at (1.2130095803074883,1.7252608807481036) {};
        \node (v2) at (1.4862313221740124,1.3442897195539345) {};
        \node (v3) at (1.8633542898207645,1.2365403002262907) {};
        \node (v4) at (2.3941405855461935,1.2481992148753074) {};
        \node (v5) at (2.7407424186315748,1.6059668807782121) {};
        \node (v6) at (2.798465321842812,2.094687461300025) {};
        \node (v7) at (2.609903838019436,2.517988751515768) {};
        \node (v8) at (2.0323256855807523,2.7700551832643514) {};
        \node (v9) at (2.609903838019436,2.517988751515768) {};
        \node (v10) at (2.0323256855807523,2.7700551832643514) {};
        \node (v11) at (2.007398523302976,2.776366667843522) {};
        \node (v12) at (1.9678545855637726,2.7838104761108013) {};
        \node (v13) at (1.9379020402511926,2.7878912102591444) {};
        \node (v14) at (1.9229370867685867,2.789475746229113) {};
        \node (v15) at (1.9099878979267313,2.7906897326830373) {};
        \node (v15) at (1.8980899846335881,2.791383495839636) {};
        \node (v16) at (1.8897547903613323,2.791499056985653) {};
        \node (v17) at (1.8802452298055947,2.791296725909999) {};
        \node (v18) at (1.4750575084633712,2.650093023923789) {};
        \node (v19) at (1.2481862049268824,2.3171044977653965) {};
        \draw[opacity=0.2,fill=gray] (v1.center)--(v2.center)--(v3.center)--(v4.center)--(v5.center)--(v6.center)--(v7.center)--(v8.center)--(v9.center)--(v10.center)--(v11.center)--(v12.center)--(v13.center)--(v14.center)--(v15.center)--(v16.center)--(v17.center)--(v18.center)--(v19.center)--(v1.center);
\draw [line width=0.4pt,color=rvwvcq] (1.2130095803074883,1.7252608807481036)-- (1.4862313221740124,1.3442897195539345);
\draw [line width=0.4pt,color=rvwvcq] (1.4862313221740124,1.3442897195539345)-- (1.8633542898207645,1.2365403002262907);
\draw [line width=0.4pt,color=rvwvcq] (1.8633542898207645,1.2365403002262907)-- (2.3941405855461935,1.2481992148753074);
\draw [line width=0.4pt,color=rvwvcq] (2.3941405855461935,1.2481992148753074)-- (2.7407424186315748,1.6059668807782121);
\draw [line width=0.4pt,color=rvwvcq] (2.7407424186315748,1.6059668807782121)-- (2.798465321842812,2.094687461300025);
\draw [line width=0.4pt,color=rvwvcq] (2.798465321842812,2.094687461300025)-- (2.609903838019436,2.517988751515768);
\draw [line width=0.4pt,color=rvwvcq] (2.609903838019436,2.517988751515768)-- (2.0323256855807523,2.7700551832643514);
\draw [line width=0.4pt,color=rvwvcq] (2.0323256855807523,2.7700551832643514)-- (2.007398523302976,2.776366667843522);
\draw [line width=0.4pt,color=rvwvcq] (2.007398523302976,2.776366667843522)-- (1.9678545855637726,2.7838104761108013);
\draw [line width=0.4pt,color=rvwvcq] (1.9678545855637726,2.7838104761108013)-- (1.9379020402511926,2.7878912102591444);
\draw [line width=0.4pt,color=rvwvcq] (1.9379020402511926,2.7878912102591444)-- (1.9229370867685867,2.789475746229113);
\draw [line width=0.4pt,color=rvwvcq] (1.9229370867685867,2.789475746229113)-- (1.9099878979267313,2.7906897326830373);
\draw [line width=0.4pt,color=rvwvcq] (1.9099878979267313,2.7906897326830373)-- (1.8980899846335881,2.791383495839636);
\draw [line width=0.4pt,color=rvwvcq] (1.8980899846335881,2.791383495839636)-- (1.8897547903613323,2.791499056985653);
\draw [line width=0.4pt,color=rvwvcq] (1.8897547903613323,2.791499056985653)-- (1.8802452298055947,2.791296725909999);
\draw [line width=0.4pt,color=rvwvcq] (1.8802452298055947,2.791296725909999)-- (1.4750575084633712,2.650093023923789);
\draw [line width=0.4pt,color=rvwvcq] (1.4750575084633712,2.650093023923789)-- (1.2481862049268824,2.3171044977653965);
\draw [line width=0.4pt,color=rvwvcq] (1.2481862049268824,2.3171044977653965)-- (1.2130095803074883,1.7252608807481036);
\draw [line width=0.4pt] (1.1221847236444418,3.272412544329517)-- (2.986487177342595,1.2791802308402918);
\draw [line width=0.4pt] (1.08827687792906,2.7286739705034733)-- (1.8802452298055947,2.791296725909999);
\draw [line width=0.4pt] (2.7407424186315748,1.6059668807782121)-- (2.5791182896401095,1.1480318486357262);
\draw (0.2824672100876302,3.2397169748420706) node[anchor=north west] {\tiny $K_{X}+\Delta=0$};
\draw (0.806696629203042,3.4600223967873494) node[anchor=north west] {\tiny $K_{X}+\Delta+H=0$};
\draw (0.806696629203042,3.6020223967873494) node[anchor=north west] {\tiny $K_{X}+\Delta<0$};
\draw (0.55,2.13) node[anchor=north west] {\tiny $\overline{\NE}(X)$};
\draw (1.3824,1.839) node[anchor=north west] {\tiny $\overline{\bNM}^l(X)$};

\begin{scriptsize}

\draw [fill=white] (2.7407424186315748,1.6059668807782121) circle (0.3pt);
\draw [fill=white] (2.798465321842812,2.094687461300025) circle (0.3pt);
\draw [fill=white] (2.609903838019436,2.517988751515768) circle (0.3pt);
\draw [fill=white] (2.0323256855807523,2.7700551832643514) circle (0.3pt);
\draw [fill=white] (2.007398523302976,2.776366667843522) circle (0.3pt);
\draw [fill=white] (1.9678545855637726,2.7838104761108013) circle (0.3pt);
\draw [fill=white] (1.9379020402511926,2.7878912102591444) circle (0.3pt);
\draw [fill=white] (1.9229370867685867,2.789475746229113) circle (0.3pt);
\draw [fill=white] (1.9099878979267313,2.7906897326830373) circle (0.3pt);
\draw [fill=white] (1.8980899846335881,2.791383495839636) circle (0.3pt);
\draw [fill=white] (1.8897547903613323,2.791499056985653) circle (0.3pt);
\draw [fill=white] (1.8802452298055947,2.791296725909999) circle (0.3pt);

\end{scriptsize}
\end{tikzpicture}
\begin{tikzpicture}[scale=2.2][line cap=round,line join=round,>=triangle 45,x=1cm,y=1cm]
\draw [line width=0.4pt] (1.08827687792906,2.7286739705034733)-- (1.0282962900315282,1.5136102356402312);
\draw [line width=0.4pt] (1.0282962900315282,1.5136102356402312)-- (1.5324096447430025,1.0171932680235873);
\draw [line width=0.4pt] (1.5324096447430025,1.0171932680235873)-- (2.5791182896401095,1.1480318486357262);
\draw [line width=0.4pt] (2.5791182896401095,1.1480318486357262)-- (3.0801446261687824,1.8185550647415993);
\draw [line width=0.4pt] (3.0801446261687824,1.8185550647415993)-- (3.0118070589036967,2.6517477117812978);
\draw [line width=0.4pt] (3.0118070589036967,2.6517477117812978)-- (2.6674908546065352,3.0065773879653963);
\draw [line width=0.4pt] (2.6674908546065352,3.0065773879653963)-- (1.853419780664326,3.1773023636669087);
\draw [line width=0.4pt] (1.853419780664326,3.1773023636669087)-- (1.08827687792906,2.7286739705034733);
\draw [line width=0.4pt] (0.8302616423644287,3.0065773879653963)-- (2.7176946558269597,1.0030773585339956);


    \node (v1) at (1.2130095803074883,1.7252608807481036) {};
        \node (v2) at (1.4862313221740124,1.3442897195539345) {};
        \node (v3) at (1.8633542898207645,1.2365403002262907) {};
        \node (v4) at (2.3941405855461935,1.2481992148753074) {};
        \node (v5) at (2.7407424186315748,1.6059668807782121) {};
        \node (v6) at (2.798465321842812,2.094687461300025) {};
        \node (v7) at (2.609903838019436,2.517988751515768) {};
        \node (v8) at (2.0323256855807523,2.7700551832643514) {};
        \node (v9) at (2.609903838019436,2.517988751515768) {};
        \node (v10) at (2.0323256855807523,2.7700551832643514) {};
        \node (v11) at (2.007398523302976,2.776366667843522) {};
        \node (v12) at (1.9678545855637726,2.7838104761108013) {};
        \node (v13) at (1.9379020402511926,2.7878912102591444) {};
        \node (v14) at (1.9229370867685867,2.789475746229113) {};
        \node (v15) at (1.9099878979267313,2.7906897326830373) {};
        \node (v15) at (1.8980899846335881,2.791383495839636) {};
        \node (v16) at (1.8897547903613323,2.791499056985653) {};
        \node (v17) at (1.8802452298055947,2.791296725909999) {};
        \node (v18) at (1.4750575084633712,2.650093023923789) {};
        \node (v19) at (1.2481862049268824,2.3171044977653965) {};
        \draw[opacity=0.2,fill=gray] (v1.center)--(v2.center)--(v3.center)--(v4.center)--(v5.center)--(v6.center)--(v7.center)--(v8.center)--(v9.center)--(v10.center)--(v11.center)--(v12.center)--(v13.center)--(v14.center)--(v15.center)--(v16.center)--(v17.center)--(v18.center)--(v19.center)--(v1.center);
\draw [line width=0.4pt,color=rvwvcq] (1.2130095803074883,1.7252608807481036)-- (1.4862313221740124,1.3442897195539345);
\draw [line width=0.4pt,color=rvwvcq] (1.4862313221740124,1.3442897195539345)-- (1.8633542898207645,1.2365403002262907);
\draw [line width=0.4pt,color=rvwvcq] (1.8633542898207645,1.2365403002262907)-- (2.3941405855461935,1.2481992148753074);
\draw [line width=0.4pt,color=rvwvcq] (2.3941405855461935,1.2481992148753074)-- (2.7407424186315748,1.6059668807782121);
\draw [line width=0.4pt,color=rvwvcq] (2.7407424186315748,1.6059668807782121)-- (2.798465321842812,2.094687461300025);
\draw [line width=0.4pt,color=rvwvcq] (2.798465321842812,2.094687461300025)-- (2.609903838019436,2.517988751515768);
\draw [line width=0.4pt,color=rvwvcq] (2.609903838019436,2.517988751515768)-- (2.0323256855807523,2.7700551832643514);
\draw [line width=0.4pt,color=rvwvcq] (2.0323256855807523,2.7700551832643514)-- (2.007398523302976,2.776366667843522);
\draw [line width=0.4pt,color=rvwvcq] (2.007398523302976,2.776366667843522)-- (1.9678545855637726,2.7838104761108013);
\draw [line width=0.4pt,color=rvwvcq] (1.9678545855637726,2.7838104761108013)-- (1.9379020402511926,2.7878912102591444);
\draw [line width=0.4pt,color=rvwvcq] (1.9379020402511926,2.7878912102591444)-- (1.9229370867685867,2.789475746229113);
\draw [line width=0.4pt,color=rvwvcq] (1.9229370867685867,2.789475746229113)-- (1.9099878979267313,2.7906897326830373);
\draw [line width=0.4pt,color=rvwvcq] (1.9099878979267313,2.7906897326830373)-- (1.8980899846335881,2.791383495839636);
\draw [line width=0.4pt,color=rvwvcq] (1.8980899846335881,2.791383495839636)-- (1.8897547903613323,2.791499056985653);
\draw [line width=0.4pt,color=rvwvcq] (1.8897547903613323,2.791499056985653)-- (1.8802452298055947,2.791296725909999);
\draw [line width=0.4pt,color=rvwvcq] (1.8802452298055947,2.791296725909999)-- (1.4750575084633712,2.650093023923789);
\draw [line width=0.4pt,color=rvwvcq] (1.4750575084633712,2.650093023923789)-- (1.2481862049268824,2.3171044977653965);
\draw [line width=0.4pt,color=rvwvcq] (1.2481862049268824,2.3171044977653965)-- (1.2130095803074883,1.7252608807481036);
\draw [line width=0.4pt] (1.1221847236444418,3.272412544329517)-- (2.986487177342595,1.2791802308402918);
\draw [line width=0.4pt] (2.798465321842812,2.094687461300025)-- (2.814507691136868,1.4630533295869779);
\draw [line width=0.4pt] (1.438621125646756,2.9340923000021446)-- (2.0323256855807523,2.7700551832643514);
\draw (0.2424672100876302,3.2297169748420706) node[anchor=north west] {\tiny $K_{X}+\Delta=0$};
\draw (0.806696629203042,3.5220223967873494) node[anchor=north west] {\tiny $K_{X}+\Delta+H=0$};
\draw (2.0824672100876302,3.4397169748420706) node[anchor=north west] {\tiny $K_{X}+\Delta+H<0$};
\draw (0.55,2.13) node[anchor=north west] {\tiny $\overline{\NE}(X)$};
\draw (1.3824,1.839) node[anchor=north west] {\tiny $\overline{\bNM}^l(X)$};
\begin{scriptsize}

\draw [fill=white] (2.798465321842812,2.094687461300025) circle (0.3pt);
\draw [fill=white] (2.609903838019436,2.517988751515768) circle (0.3pt);
\draw [fill=white] (2.0323256855807523,2.7700551832643514) circle (0.3pt);

\end{scriptsize}
\end{tikzpicture}
  \caption{}\label{fig1}
\end{figure}

See Section \ref{sec-prelim} for the definition of numerical pull-back of curves.
The statements (1) and (2) of Theorem \ref{main theorem} imply that the rays $R_i\in\Sigma^l(X,\Delta)$ can accumulate only toward the hyperplanes supporting simultaneously both cones $\overline{\NE}(X)_{K_X+\Delta\geq 0}$ and $\overline{\bNM}^l(X)$ and the rays $R_i$ are discrete away from these hyperplanes. See Figure \ref{fig1}.

\medskip

Due to the complication, we briefly outline here the strategy of the proof of our main result, Theorem \ref{main theorem}. The detailed proof is given in Section \ref{sect-proofs}.
As in Theorem \ref{main theorem}, let $X$ be a $\Q$-factorial normal projective variety and $(X,\Delta)$ be a klt pair.\\

\noindent\textbf{STEP1}: Let $H$ be any ample divisor on $X$.
In Proposition \ref{prop-main prop}, we first prove that the set $\Sigma^l(X,\Delta+H)$ of $l$-th coextremal rays of $(X,\Delta+H)$ is finite. This immediately implies (2) of Theorem \ref{main theorem}. In the proof of Proposition \ref{prop-main prop}, the finiteness of minimal models (Theorem \ref{thrm-finite models}) and generalized rationality theorem (Theorem \ref{g-rationality thm}) play crucial roles.

\bigskip

\noindent\textbf{STEP2}:
For any real number $\eps$ such that $0<\eps\leq 1$, it is easy to observe that $\Sigma^l(X,\Delta+\eps H)\supseteq\Sigma^l(X,\Delta+\eps H)_{K_X+\Delta+H<0}$.
Therefore, as a byproduct of \textbf{STEP1} (or Proposition \ref{prop-main prop}), we obtain that the set $\Sigma^l(X,\Delta+\eps H)_{K_X+\Delta+H<0}$ of $(K_X+\Delta+H)$-negative extremal rays of the cone $\overline{\NE}^l(X,\Delta+\eps H)$ is also finite.
This finiteness implies the following (Corollary \ref{cor-(1-e) finite}):

$$
\overline{\NE}^l(X,\Delta+\eps H)=\overline{\NE}^l(X,\Delta+\eps H)_{K_X+\Delta+H\geq 0}+\sum_{R_i\in \Sigma^l(X,\Delta+\eps H)_{K_X+\Delta+H<0}} R_i.
$$

\medskip

\noindent\textbf{STEP3}: The intersection of the above equality over all $\eps$ satisfying $0<\eps\leq 1$ gives
$$
\bigcap_{0<\eps\leq 1} \;\overline{\NE}^l(X,\Delta+\eps H)=\overline{\NE}^l(X,\Delta)_{K_X+\Delta+H\geq 0}+
\overline{\sum_{R_i\in\Sigma^l(X,\Delta)_{K_X+\Delta+H<0}} R_i}.
$$
Since the choice of the ample divisor $H$ was arbitrary and
$$
\overline{\NE}^l(X,\Delta)=\bigcap_{0<\eps\leq 1} \;\overline{\NE}^l(X,\Delta+\eps H),
$$
we finally obtain the statement (1) of Theorem \ref{main theorem}.

\medskip

\noindent\textbf{STEP4}:
For each ray $R_i\in\Sigma^l(X,\Delta+H)$, we can find an ample divisor $A_i$ on $X$ such that $R_i=\overline{\NE}^l(X,\Delta+H)\cap\{\eta\in N_1(X)|\eta\cdot (K_X+\Delta+H+A_i)=0\}$.
Since the bounding divisor $K_X+\Delta+H+A_i$ is pseudoeffective,
if we run the MMP on $(X,\Delta+H+A_i)$ with scaling of some large multiple of $A_i$, then we obtain a birational map $\varphi:X\dashrightarrow X'$ such that $\varphi_*(K_X+\Delta+H+A_i)=K_{X'}+\Delta'+H'+A'_i$ is nef.
The Cone Theorem implies that there exists a rational curve $C'$ on $X'$ which is trivial with respect to $K_{X'}+\Delta'+H'+A'_i$ and its numerical pull back to $N_1(X)$ spans the ray $R_i$.
Hence, we obtain (3) of Theorem \ref{main theorem}.

\medskip

We remark first of all that in Theorem \ref{thrm-batyrev} and in our result (Theorem \ref{main theorem}), the enlargement of $\overline{\bNM}^l(X)$ by the half cone $\overline{\NE}(X)_{K_X+\Delta\geq 0}$ is necessary.
Due to the enlargement, all the $(K_X+\Delta)$-negative extremal rays of $\overline{\NE}^l(X,\Delta)$ can be supported by the hyperplane of the form  $\{\eta|\eta\cdot(K_X+\Delta+H)=0\}$ for some ample divisor $H$. Thus we are able to run the MMP with scaling on $(X,\Delta+H)$ assuming that the pair is klt. Without such enlargement, some $(K_X+\Delta)$-negative extremal rays of $\overline{\bNM}^l(X)$ can be only supported by the divisor of the form $K_X+\Delta+B=0$ for a big divisor $B$. If this is the case, there is no guarantee that $(X,\Delta+B)$ is klt or that we can run the MMP on the pair.
Note also that the cone $\overline{\bNM}^l(X)$ for $l\geq 1$ can have circular boundary in the $(K_X+\Delta)$-negative half space in general. See \cite[Example 4.9]{Lehmann-cone} for such example.

Lastly, we also remark that a slightly stronger version of Theorem \ref{thrm-batyrev} was proved by Batyrev in \cite{Batyrev} for terminal threefolds $X$ with $\Delta=0$ using the boundedness of terminal Fano threefolds. See Remark \ref{rmk-Batyrev} for details.
The recent result on the conjecture of Borisov-Alexeev-Borisov (BAB) by Birkar (\cite{B1}, \cite{B2}) can be also applied to obtain a stronger structure theorem of $\overline{\NE}^0(X,\Delta)$ in higher dimensions. See \cite{Araujo} for details.
However, as of now, it is unclear what kind of boundedness condition is needed to improve our result to the stronger form for the cases $0< l <\dim X$.

In Section \ref{sec-prelim}, we first recall some basic fact that we need in the subsequent sections.
In Section \ref{sec-general rationality}, we give a generalization of the well known Rationality Theorem which will play crucial roles in the proof of our main result.
In Section \ref{sect-proofs}, after collecting some more ingredients we complete the proof of the main result (Theorem \ref{thrm-batyrev}).

\medskip

\section*{Aknowledgement}
The first author was partially supported by NRF-2016R1C1B2011446. Some parts of this paper were written during the first author's visits to University of Tokyo, Osaka University and IBS-Center for Geometry and Physics. He is grateful for their hospitality.  The second author was partially supported by JSPS KAKENHI $\#$15H03611, 16H02141, 17H02831, and 18H01108. He thanks Professors Tommaso de Fernex, Masayuki Kawakita, Janos Koll\'ar, and Keiji Oguiso  for comments, questions, and  discussion. He brushed up this paper when he stayed at MSRI, Johns Hopkins University,  National Taiwan University, and University of Utah. He is grateful to it for their hospitality.

\section{Preliminaries}\label{sec-prelim}

We work over an algebraically closed field of characteristic $0$, e.g., the field $\mathbb C$ of complex numbers.
By a variety $X$, we mean a normal projective variety and divisors are always $\mathbb R$-divisors unless otherwise stated. A \emph{pair} $(X,\Delta)$ consists of a variety $X$ and a divisor $\Delta$ on $X$ such that $K_X+\Delta$ is $\R$-Cartier.

\subsection{Minimal model program with scaling}

We use the standard terminologies in the minimal model program (MMP), for example, as in \cite{KM}.
Below, we will recall some of them.

\begin{definition}[Singularities of pairs]
For a given pair $(X,\Delta)$ with an effective divisor $\Delta$ on a variety $X$, let $f:Y\to X$ be a log resolution.
Then we can write
$$
K_Y=f^*(K_X+\Delta)+\sum a(E_i;X,\Delta)E_i
$$
for some prime divisors $E_i$ on $Y$.
The pair $(X,\Delta)$ is said to have
\begin{enumerate}
  \item \emph{kawamata log terminal} (klt) singularities if $a(E_i;X,\Delta)>-1$ for all $i$ and
  \item \emph{log canonical} (lc) singularities if $a(E_i;X,\Delta)\geq -1$ for all $i$.
\end{enumerate}
\end{definition}

\begin{definition}\label{def-min model}
Let $(X,\Delta)$ be a klt pair such that $K_X+\Delta$ is pseudoeffective.
Let $\varphi:X\dashrightarrow X'$ be a birational modification of $X$ and let $\Delta'=\varphi_*\Delta$.
The pair $(X',\Delta')$ is called a \emph{minimal model of} $(X,\Delta)$ if
\begin{enumerate}
\item $(X',\Delta')$ is a $\Q$-factorial klt pair,
\item $K_{X'}+\Delta'$ is nef, and
\item $\varphi$ is $(K_X+\Delta)$-negative, that is, the inequality $a(E_i;X,\Delta)\leq a(E_i;X',\Delta')$ holds for all prime divisors $E_i$ over $X$ and it is strict for any prime divisor $E_i$ on $X$ which is exceptional on $X'$.
\end{enumerate}
\end{definition}

Let $\Amp_1(X)$ be the ample cone and $\overline{\Amp}_1(X)$ the nef cone in $N^1(X)$.
Let $E,D$ be divisors on $X$ such that $E+rD$ is nef for some $r\geq0$.
The following threshold will be useful below:
$$
r_1=r_1(E;D):=\sup\{t\;|\;E+tD\in\overline{\Amp}_1(X)\}.
$$
(We will often call $r_1(E;D)$ the nef threshold.)

\bigskip

We explain how to run the minimal model program (with scaling) on a given $\Q$-factorial klt pair $(X,\Delta)$.
Let $H$ be a divisor on $X$ such that $(X,\Delta+H)$ is klt and $K_X+\Delta+H$ is nef.

To run the minimal model program on $(X,\Delta)$ with scaling of $H$, we find a sequence of birational modifications of $X$ as follows.
If $K_X+\Delta$ is nef, then $(X,\Delta)$ is already a minimal model and there is nothing further to do.
Thus we assume that $K_X+\Delta$ is not nef. Let $(X_0,\Delta_0)=(X,\Delta)$ and $H_0=H$.
Since $K_{X_0}+\Delta_0+H_0$ is nef, we have $t_0:=r_1(H_0;K_{X_0}+\Delta_0)\geq 1$.
Therefore, $(X_0,\Delta_0+\frac{1}{t_0}H_0)$ is klt and $K_{X_0}+\Delta_0+\frac{1}{t_0}H_0$ is nef.
Suppose that inductively for $i\geq 1$ we have obtained a $(K_{X_{i-1}}+\Delta_{i-1})$-negative birational map $\varphi_i:X_{i-1}\dashrightarrow X_i$ such that $(X_i,\Delta_i)$ is a $\Q$-factorial klt pair with $\Delta_i=\varphi_{i*}\Delta_{i-1}$. Suppose that we have also constructed a divisor $H_i$ on $X_i$ such that $t_i:=r_1(H_i;K_{X_i}+\Delta_i)\geq 1$.

If $K_{X_i}+\Delta_i$ is nef, then $(X_i,\Delta_i)$ is a minimal model.
Thus we assume that $K_{X_i}+\Delta_i$ is not nef.
By the Contraction Theorem (\cite[Theorem 3.7]{KM}), there exists a $(K_{X_i}+\Delta_i)$-negative extremal ray $R$ in $\overline{\NE}(X_i)$ which is trivial with respect to the nef divisor $K_{X_i}+\Delta_i+\frac{1}{t_i}H_i$. If the associated extremal contraction $f:X_i\to X_i'$ is not birational, then $f$ is called a \emph{Mori fibre space} and the MMP comes to an end. If $f$ is birational, then it is either divisorial or small. If $f$ is a divisorial contraction, then we let $\varphi_{i+1}:=f$, $(X_{i+1},\Delta_{i+1}):=(X'_i,f_{i*}\Delta_i)$ and $H_{i+1}:=f_*H_i$. If $f$ is small, we take a \emph{flip} $f^+:X_i\dashrightarrow X_i^+$ over $X'_i$ which is uniquely associated to $R$.
After taking this flip, we denote $\varphi_{i+1}=f^+$, $(X_{i+1},\Delta_{i+1})=(X_i^+,\varphi_{i*}\Delta_i)$ and $H_{i+1}:=\varphi_{(i+1)*}H_i$. In either case, $(X_{i+1},\Delta_{i+1})$ is a $\Q$-factorial klt pair and $t_{i+1}=r_1(H_{i+1};K_{X_{i+1}}+\Delta_{i+1})\geq 1$.

If $\varphi_i:X_{i-1}\dashrightarrow X_{i}$ are the birational maps obtained as above, their composition
\begin{equation}\label{mmp}
\xymatrix{
X=X_0 \ar@{-->}[r]^{\varphi_1} &X_1 \ar@{-->}[r]^{\varphi_2} & X_2\ar@{-->}[r]^{\varphi_3} &\cdots \ar@{-->}[r]^{\varphi_{i}}& X_i \ar@{-->}[r]^{\varphi_{i+1}}& X_{i+1}\ar@{-->}[r] &\cdots
}
\end{equation}
is called the \emph{MMP on $(X,\Delta)$ with scaling of $H$}.
The composition $\varphi_i\circ\varphi_{i-1}\circ\cdots\circ\varphi_1$ for any $i\geq 1$ is called a \emph{partial MMP}.
If, for some sufficiently large $i>0$, the pair $(X_i,\Delta_i)$ is a minimal model or a Mori fibre space of $(X,\Delta)$ so that the process stops, then we say that the MMP with scaling terminates.

\begin{remark}\label{rem-mmp}
By construction, it is easy to see that $\{t_i\}_{i\geq 0}$ is an increasing sequence $$t_0\leq t_1\leq t_2\leq \cdots.$$
Note also that the difference $t_{i+1}-t_i$ can be expressed as the nef threshold
$$
t_{i+1}-t_{i}=r_1(\widehat{H}_{i+1};K_{X_{i+1}}+\Delta_{i+1})
$$
where $\widehat{H}_{i+1}:=\varphi_{i+1*}(H_i+t_{i}(K_{X_i}+\Delta_i))$ is a nef divisor.
\end{remark}

\begin{theorem}[{\cite{bchm}}]\label{thrm-term'n MMP scaling}
Let $(X,\Delta)$ be a $\Q$-factorial klt pair with a big divisor $\Delta$ on $X$. Then the MMP on $(X,\Delta)$ with scaling terminates.
\end{theorem}

We can also easily deduce the existence of a minimal model if $K_X+\Delta$ is big.

\begin{proposition}\label{prop-+-eps MMP}
Let $(X,\Delta)$ be a $\Q$-factorial klt pair and $E$ an effective divisor on $X$ such that $(X,\Delta+E)$ is klt and $K_X+\Delta+E$ is nef.
Let $\varphi:X\dashrightarrow X'$ be the MMP on $(X,\Delta)$ with scaling of $E$, assuming that such MMP terminates.
Then the following hold.
\begin{enumerate}
\item For any sufficiently small $\eps>0$, $\varphi$ is also the MMP on $(X,\Delta+\eps E)$ with scaling of $(1-\eps)E$.

\item If $\Supp E\subseteq\Supp\Delta$, then $\varphi$ is a partial MMP on     $(X,\Delta-\eps E)$ with scaling of $(1+\eps)E$ for sufficiently small $\eps>0$.
\end{enumerate}
\end{proposition}
\begin{proof}
(1) It is clear that the MMP on $(X,\Delta+\eps E)$ with scaling of $(1-\eps)E$ can be obtained as a partial MMP on $(X,\Delta)$ with scaling of $E$.

(2) Clear by the definition of the MMP with scaling.
\end{proof}

We will need the following result on the finiteness of the minimal models.
\begin{theorem}\label{thrm-finite models}
Let $(X,\Delta)$ be a $\Q$-factorial klt pair and $A$ an effective ample $\Q$-divisor on $X$ such that $(X,\Delta+A)$ is klt.
For any effective ample divisor $H$ such that $(X,\Delta+A+H)$ is klt, there exists a convex cone in $\text{WDiv}(X)_{\R}$ whose relative interior $U$ consists of ample divisors including $H$ and such that there are finitely many isomorphism classes of minimal models $(X,\Delta')$ for all $\Delta'\in\mathfrak{D}$ where
$$
\mathfrak{D}=\{\Delta'=\Delta+A+H'| \text{ $H'\in U$ and $(X,\Delta')$ is klt}\}.
$$
We may assume that $\dim U=\dim N^1(X)$ and the image of $U$ in $N^1(X)$ is an open subcone of the ample cone $\Amp_1(X)$.
\end{theorem}
\begin{proof}
Let $\Delta=\sum_{i=1}^m d_i\Delta_i$ where $\Delta_i$ are prime divisors. Since $H$ is ample, we can take a set of effective ample $\Z$-divisors $\{H_1,H_2,\cdots,H_{\rho}\}$ such that $H$ is contained in the cone $U=\oplus_{j=1}^{\rho}\R_{>0} H_j$. Since $A$ is ample, we may also assume that $A$ has no common prime components with the divisors $\Delta_1,\Delta_2,\cdots,\Delta_m, H_1,\cdots,H_\rho$.

It is enough to prove that this open cone $U$ satisfies the required property. Consider the finite dimensional vector space $\mathcal L=\{\sum_{i=1}^m a_i\Delta_i+\sum_{j=1}^\rho b_jH_j\;|\;a_i,b_j\in\R\}$ and the affine space $V=A+\mathcal L$.
Then by \cite[Corollary 1.1.5]{bchm}, there are finitely many isomorphism classes of minimal models of the pairs $(X,\Delta')$ for all $\Delta'\in\mathfrak{D}'$ where
$$
\mathfrak{D}'=\{\Delta'\in V\;|\; (X,\Delta') \text{ is klt }\}.
$$
Since we have $\mathfrak{D}\subseteq\mathfrak{D}'$ by construction, the statement follows.

By taking the effective ample $\Z$-divisors $H_j$ such that $\{H_1,H_2,\cdots,H_{\rho}\}$ forms a basis of $N^1(X)$, we may assume that $U$ considered as a cone in $N^1(X)$ is an open subcone of the ample cone $\Amp_1(X)$.
\end{proof}

\subsection{$N^1(X)_{\R}$}\hfill

For a divisor $D$ on $X$, let $\bp(D)$ be the \emph{augmented base locus} which is  defined as the stable base locus $\text{SB}(D-A)$ for some sufficiently small ample divisor $A$ such that $D-A$ is a $\Q$-divisor.
Similarly, the \emph{diminished base locus} $\bm(D)$ is defined as the union $\cup_{A}\text{SB}(D+A)$ for all ample divisors $A$ such that $D+A$ are $\Q$-divisors.
By construction, $\bp(D)$ is always a Zariski closed subset of $X$ while $\bm(D)$ can be non-Zariski closed \cite{J on B-}.
It is also well known that both $\bp(D)$ and $\bm(D)$ depend only on the numerical class $\eta=[D]$ in $N^1(X)_{\R}$.

\begin{definition}
Let $l$ be an integer such that $1\leq l\leq d=\dim X$.
We define the \emph{$l$-ample cone} $\Amp_l(X)$ in $N^1(X)$ as
$$
\Amp_l(X):=\{\eta\in N^1(X) \;|\dim\bp(\eta)< l\}.
$$
\end{definition}

It is clear that the $1$-ample cone $\Amp_1(X)$ is the ample cone $\Amp_1(X)$ and
the $d$-ample cone $\Amp_d(X)$ is the big cone which is denoted by $\Bigdiv(X)$.
It is also easy to see that by Lemma \ref{lem-property B-B+} (1), the $l$-ample cone $\Amp_l(X)$ is an open convex cone and satisfies $\Amp_l(X)\subseteq\Amp_{l'}(X)$ for $l\leq l'$.

Let $H$ be a divisor such that $H\in\overline{\Amp}_i(X)$ for some integer $i$ such that $1\leq i\leq d$.
For any divisor $D$ on $X$, we define the following threshold
$$
r_i(H;D):=\sup\{t\in \mathbb R_{\geq0}\;|\; H+tD \in \overline{\Amp}_{i}(X)\}.
$$
Unless $D\in\overline{\Amp}_i(X)$, we have $r_i(D;H)<+\infty$.
It is easy to see that $0\leq r_{i}(H;D)\leq r_{i+1}(H;D)$ since $\overline{\Amp}_{i}(X)\subseteq\overline{\Amp}_{i+1}(X)$.
The following statements are rather obvious, and we will use them repeatedly.

\begin{lemma}\label{lem-not nef}
For a fixed integer $i$ such that $1\leq i\leq d=\dim X$, the following hold:
\begin{enumerate}
\item Let $D$ be a nef divisor and $D'$ be any divisor on $X$. Then $r_1(D;D')\leq r_i(D;D')$ and the equality holds if and only if $D+r_i(D;D')D'$ is nef.
\item Let $H\in\overline{\Amp}_i(X)$ and $F\not\in\overline{\Amp}_i(X)$.
Then $H+rF$ is not nef for any $r>r_{i}(H;F)$.
\end{enumerate}
\end{lemma}
\begin{proof}
$(1)$ Obvious.

$(2)$ By definition, we have $H+rF\not\in\overline{\Amp}_i(X)$ for any $r>r_i(H;F)$.
Since $\Nef(X)=\overline{\Amp}_1(X)\subseteq\overline{\Amp}_i(X)$, clearly $H+rF$ is not nef.
\end{proof}

\begin{lemma}\label{lem-property B-B+}
Let $D,D'$ be arbitrary divisors on $X$ (that are not necessarily effective).
\begin{enumerate}
  \item $\bp(D)\supseteq\bp(D+\eps D')$ for any sufficiently small $\eps>0$.
  \item Suppose that $H$ is a nef divisor.
  Then $\bp(D+H)\subseteq\bp(D)$ and $\bp(D)\subseteq\bp(D-\eps H)$ for any sufficiently small $\eps>0$.
\end{enumerate}
\end{lemma}
\begin{proof}
(1)  This is Corollary 1.6 of \cite{elmnp-asymptotic inv of base}.

\noindent (2) First, we prove $\bp(D+H)\subseteq\bp(D)$. We may assume that $D$ is big.
For a fixed ample divisor $A$ and a sufficiently small $\eps>0$, we have
$$
\bp(D+H)=\bp(D-\eps A+H+\eps A)\subseteq\bp(D-\eps A)\cup\bp(H+\eps A)=\bp(D-\eps A)
$$
where $\bp(H+\eps A)=\emptyset$ since $H+\eps A$ is ample.
Furthermore, since $\bp(D-\eps A)=\bp(D)$, we obtain the first inclusion.
To prove the second inclusion, let $\eps>0$ be sufficiently small.
If $0<\eps'\ll\eps$, then we have
\begin{align*}
  \bp(D-\eps H)&\supseteq\bp(D-\eps H-\eps' A) \\
   & =\bp(D-(\eps H+\eps' A))=\bp(D).
\end{align*}
\end{proof}

\begin{lemma}\label{lem-boundary Amp_l}
Let $D$ be a divisor such that $D\in\partial\overline{\Amp}_l(X)$ for an integer $l$ such that $1\leq l\leq \dim X$. Then there exists an irreducible component $V\subseteq\bp(D)$ of dimension $\geq l$ such that $V\not\subseteq\bm(D)$.
\end{lemma}
\begin{proof}
Since $\Amp_l(X)$ is an open cone, we have $D\not\in\Amp_l(X)$ and there exists an irreducible component $V\subseteq\bp(D)$ such that $\dim V\geq l$.
If $V\subseteq\bm(D)$, then there exists an ample divisor $A$ such that $V\subseteq\bm(D+A)$. Thus $V\subseteq\bp(D+A)$. However, since $\Amp_1(X)\subseteq\Int\overline{\Amp}_l(X)$ and $A\in\Amp_1(X)$, we have $D+A\in\Amp_l(X)$. This is a contradiction since $V\subseteq\bp(D+A)$.
Therefore, we must have $V\not\subseteq\bm(D)$.
\end{proof}

\medskip

\subsection{$N_1(X)_{\R}$}\hfill

Suppose that we are given a birational map $\varphi:X\dashrightarrow X'$ between $\Q$-factorial normal projective varieties. Suppose also that $\varphi^{-1}$ does not contract any divisors.
Then the pull-back of divisors define a linear map $\varphi^*:N^1(X')\to N^1(X)$ such that the composition $\varphi_*\circ\varphi^*:N^1(X')\to N^1(X')$ with the obvious push-forward linear map $\varphi_*:N^1(X)\to N^1(X')$ is the identity.
For the case of curves, since we cannot take the pull-back of curves freely, the similar pull-back procedure for divisors just explained is not possible. We can only define them as a dual linear map of $\varphi_*$ as follows.

\begin{definition}\label{def-num pullback}
Under the same notations as above, let $\varphi_*:N^1(X)\to N^1(X')$ be the linear map defined by the push-forward of divisors. Then the \textit{numerical pull-back of curves} $\varphi_1^*:N_1(X')\to N_1(X)$ is defined as the dual linear map of $\varphi_*:N^1(X)\to N^1(X')$.
\end{definition}

Note that for a given class $\eta$ of $N_1(X)$, it is unclear in general how to find a curve on $X$  representing the class $\eta$.  Even if we are able to do so, we find that such representation is rather unexpected. For example, the numerical pull-back of an effective integral curve need not be integral nor effective. See Examples 4.2 and 4.3 in \cite{Araujo} for details. Nonetheless, an advantage of the numerical pull-back $\varphi_1^*$ is that at least in the situations we treat below, we can express a class $\eta\in N_1(X)$ in terms of curves on some $X'$ which is birational to $X$.

\medskip

We will need the following result.

\begin{theorem}\label{thrm-l-nef criterion}
Let $(X,\Delta)$ be a $\Q$-factorial klt pair. Assume that for an integer $l$ such that $0\leq l< d=\dim X$, we have $\overline{\NE}^{l}(X,\Delta)_{K_X+\Delta<0}\neq \emptyset$ (or equivalently, $\Sigma^l(X,\Delta)\neq\emptyset$).
Then $K_X+\Delta\not\in\overline{\Amp}_{d-l}(X)$.
\end{theorem}
\begin{proof}
By \cite[Theorem 4.3]{C}, we have $\overline{\NE}^l(X,\Delta)=\overline{\NE}(X)_{K_X+\Delta\geq 0}+\overline{\Amp}_{d-l}(X)^{\vee}$, where $\overline{\Amp}_{d-l}(X)^{\vee}\subseteq N_1(X)$ denotes the dual of the cone $\overline{\Amp}_{d-l}(X)$.
If $\overline{\NE}^l(X,\Delta)_{K_X+\Delta<0}\neq \emptyset$, then this in particular implies that $K_X+\Delta$ is negative on some part of the dual cone $\overline{\Amp}_{d-l}(X)^{\vee}$. Therefore, we have $K_X+\Delta\not\in\overline{\Amp}_{d-l}(X)$.
\end{proof}

If $(X,\Delta)$ is a klt pair and $R$ is a $(K_X+\Delta)$-negative extremal ray of $\overline{\NE}^l(X,\Delta)$, i.e., $R\in\Sigma^l(X,\Delta)$, then by the expression
\begin{equation}\label{expanded Mori cone}
\overline{\NE}^l(X,\Delta)=\overline{\NE}(X)_{K_X+\Delta\geq0}+\overline{\sum_{X\dashrightarrow X'}\overline{\NM}^l(X,X')}
\end{equation}
and the discreteness of the rays in $\Sigma^l(X,\Delta)$ (which we will prove below), $R$ can be considered as an extremal ray of $\overline{\NM}^l(X')$ for some small $\Q$-factorial modification $X\dashrightarrow X'$. If the ray $R$ as an extremal ray of $\overline{\NM}^l(X')$ is spanned by a curve $C'$ on $X'$, then the ray $R$ as an extremal ray in $\Sigma^l(X,\Delta)$ is spanned by the numerical pull-back $\varphi_1^*(C')$ of $C'$ on $X'$ by definition. Note also that the small $\Q$-factorial modifications $\varphi:X\dashrightarrow X'$ in the above expression can be realized as the MMP on $(X,\Delta)$ with scaling.

\section{Generalized Rationality Theorem}\label{sec-general rationality}
In this section, we prove a generalization of the Rationality Theorem (Theorem \ref{g-rationality thm}). We will present the result in the form which is suitable for our application using the finiteness of minimal models (Theorem \ref{thrm-finite models}).
Its immediate consequence is Proposition \ref{prop-main prop}, which is the key result used in STEP1 of the proof of Theorem \ref{main theorem} as in Introduction.

We first recall the following Rationality Theorem which played a crucial role in the proof of the Cone Theorem (Theorem \ref{thrm-Cone theorem}).

\begin{theorem}[{Rationality Theorem, \cite[Theorem 3.5]{KM}}]\label{thrm-rationality}
Let $(X,\Delta)$ be a $\Q$-factorial klt pair with an effective $\Q$-divisor $\Delta$ on $X$.
Suppose that $K_X+\Delta$ is not nef and let $H$ be a nef and big Cartier divisor.
Then there is a constant $M(X,\Delta)>0$ depending only on the pair $(X,\Delta)$ such that
$$r_1(H;K_X+\Delta):=\sup\{t\geq 0|H+t(K_X+\Delta)\in\overline{\Amp}_1(X)\}$$
 is a rational number of the form $\frac{\alpha}{\beta}$ for some positive integers $\alpha,\beta$ such that $0< \beta< M(X,\Delta)$.
\end{theorem}

We will need a generalization of this result in the proof of our main result, Theorem \ref{main theorem}.

For the cones $R$ and $V$ in $\R^\rho$, we denote
$R+V:=\{r+v|r\in R,\; v\in V\}$.

\begin{theorem}\label{g-rationality thm}Let $(X,\Delta+A)$ be a $\Q$-factorial klt pair with some effective $\Q$-divisor $\Delta$ and effective ample $\Q$-divisor $A$. Let $l$ be a positive integer. Assume that  $K_X+\Delta+A \not \in \overline{\mathrm{Amp}}_l(X)$. Then for a $\mathbb{Q}$-divisor $H \in (\mathbb{R}_{\geq 0} [K_X+\Delta+A]+ \mathrm{Amp}_1(X)) \cap \overline{\mathrm{Amp}}_l(X)$, there exists an open convex cone $U \subseteq  \mathrm{Amp}_1(X)$ such that $H \in (\mathbb{R}_{\geq 0} [K_X+\Delta+A]+ U) \cap \overline{\mathrm{Amp}}_l(X)$. Moreover for any Weil divisor $H' \in (\mathbb{R}_{\geq 0} [K_X+\Delta+A]+ U) \cap \overline{\mathrm{Amp}}_l(X)$, the threshold  $$r_l(H';K_X+\Delta+A)=\sup\{t| H'+t(K_X+\Delta+A) \in \overline{\mathrm{Amp}}_l(X)\}
$$ is a rational number whose positive denominator can be bounded from above by some constant $M(U,X,\Delta+A)$ depending only on the pair $(X,\Delta+A)$ and $U$.
\end{theorem}

\begin{proof}
By the condition that $H \in (\mathbb{R}_{\geq 0} [K_X+\Delta+A]+ \mathrm{Amp}_1(X)) \cap \overline{\mathrm{Amp}}_l(X)$, there exists some $a\geq 0$ such that $G_0:=H-a(K_X+\Delta+A)$ is ample. Moreover by taking $s\gg 0$, we may assume that  $K_X+\Delta+A+sG_0$ is ample. For $G:=sG_0$,
there exists some open subcone $U\subseteq\Amp_1(X)$ which satisfies the statements of Theorem \ref{thrm-finite models}. Note that  these $G_0,s,G,U$ depend only on $H, X,\Delta+A$.
On the other hand, since $U$ is a convex cone and $H=a(K_X+\Delta+A)+\frac{1}{s}G$,
$$H \in \mathbb{R}_{\geq 0} [K_X+\Delta+A]+ U$$
holds. Moreover by shrinking $U$ if necessary, we may assume that for a  $\mathbb{Z}$-divisor $H'\in U$, we have some constant $k=k(U,X,\Delta+A)>0$ such that $K_X+\Delta+A+kH'$ is ample. This follows from the fact $X$ has only $\mathbb{Q}$-factorial rational singularities and from the length of extremal rays \cite{kawamata-length}. Here, $k$ dose not depend on $H'$. By the equality of $r_l(kH';K_X+\Delta+A )=kr_l(H';K_X+\Delta+A)$, we can replace $H'$ with $kH'$. Thus we may assume that  $K_X+\Delta+A+H'$ is ample. Therefore  we have $r_l:=r_l(H';K_X+\Delta+A)>1$.

Here we let $\Delta':=\Delta+A+\frac{1}{r_l}H'$, and $S:=(1-\frac{1}{r_l})H'$.
Now we run $(K_X+\Delta')$-MMP with scaling of $S$:
$$
\xymatrix{
X=X_0 \ar@{-->}[r]^{\varphi_1} &X_1 \ar@{-->}[r]^{\varphi_2} & X_2\ar@{-->}[r]^{\varphi_3} &\cdots \ar@{-->}[r]^{\varphi_{m}}& X_m.
}
$$
In fact, this MMP is also $(K_X+\Delta+A)$-MMP with scaling by  $H'$ and this MMP terminates by \cite{bchm}. Now we use the same notations as in Subsection 2.1. By Remark \ref{rem-mmp}, for $i\geq 0$, i.e. $t_i=r_1(H'_i; K_{X_i}+\Delta'_i)$, where $\Delta_i, A_i, H'_i, \Delta'_i, S_i$ are strict transforms of  $\Delta, A, H', \Delta', S$ on $X_i$, respectively. Set  $D_{i+1}:=\varphi_{i*}(H'_i+t_i(K_{X_i}+\Delta_i+A_i))$. Then we have
$$t_{i+1}-t_i=r_1(D_{i+1}; K_{X_{i+1}}+\Delta_{i+1}+A_{i+1}).
$$

Now we  will first show that for each $i$, the denominator of $t_i$ can be bounded from above by some constant $M_{i}=M_{i}(U,X, \Delta+A) $ by induction on $i$.  First of all, since $r_1(D_{i+1}; K_{X_{i}}+\Delta_{i}+A_{i})$ is a nef threshold, Theorem \ref{thrm-rationality} implies that it is a rational number whose positive  denominator can be bounded  from above by some constant  $N_{i}=N_{i}(X, \Delta+A,t_{i-1})$. Since $X_i$  has only $\mathbb{Q}$-factorial rational singularities, Theorem \ref{thrm-finite models} and Theorem \ref{thrm-rationality} imply that we may take $N_i=N_i(U,X, \Delta+A,t_{i-1})=N_{i}(U,X, \Delta+A)$ which depends only on the cone $U$ and the pair $(X,\Delta+A)$. Note also that the positive  denominator of $t_0=r_1(H'; K_X+\Delta)$ can  be also bounded from above by some constant $M_{0}=M_{0}(X, \Delta+A)$ since
$X$ has only $\mathbb{Q}$-factorial rational singularities.
\\
To prove by induction, we assume that the denominator of  $t_{i-1}$ can be bounded from above by some constant $M_{i-1}=M_{i-1}(U,X, \Delta+A)$.
Then, by the equation
$$t_{i}-t_{i-1}=r_1(D_{i}; K_{X_{i}}+\Delta_{i}+A_{i}),
$$
the denominator of  $t_{i}$ can be also bounded by the constant $M_{i}:=M_{i-1}N_i$ which depends only on the cone $U$ and the pair $(X, \Delta+A)$.

Next we compare $t_{m-1}$ and $r_{l}$. In general, it holds that $r_l \geq t_{m-1}$. We will prove this first. By Lemma \ref{lem-boundary Amp_l}, $K_X+\Delta+A+\frac{1}{r_l}H' \in \partial \overline{\mathrm{Amp}_l}$ implies that $\mathbf{B}_{+}(K_X+\Delta+A+\frac{1}{r_l}H' )$ has some irreducible component $V$ of dimension $\geq l$ such that  $V \not \subseteq \bm(K_X+\Delta+A+\frac{1}{r_l}H' )$.  Thus an MMP $X\dashrightarrow X_m$ is isomorphic at the generic point of $V$. Denote the strict transform of $V$ on $X'$ by $V'$. Then $V' \subseteq \bp (K_{X_m}+\Delta_m+A_m+\frac{1}{r_l}H_m' )$. Now by contradiction, we assume that $r_l < t_{m-1}$. Then we have some $a$ such that $t_{m-1}\gg a>r_l$ and $\bm (K_{X_m}+\Delta_m+A_m+\frac{1}{a}H'_m )=\mathbf{B}_{+}(K_{X_m}+\Delta_m+A_m+\frac{1}{r_l}H'_m )$.  Thus by $\mathbf{B}_{-}(K_{X_m}+\Delta_m+A_m+\frac{1}{a}H'_m) \supseteq V'$, in particular $\bm (K_{X_m}+\Delta_m+A_m+\frac{1}{a}H'_m) \not = \emptyset$, i.e. $K_{X_m}+\Delta_m+A_m+\frac{1}{a}H'_m$ is not nef. This is a contradiction to the fact that  $K_{X_m}+\Delta_m+A_m+\frac{1}{r_l}H'_m$ and $K_{X_m}+\Delta_m+A_m+\frac{1}{t_{m-1}}H'_m$ are nef.  Thus we can conclude that $r_l \geq t_{m-1}$. 

Since Theorem \ref{g-rationality thm} is already proved  when $r_l = t_{m-1}$, we may assume that  $r_l > t_{m-1}$. Since $K_{X_m}+\Delta_m+A_m+\frac{1}{t_{m-1}}H'_m$ is nef,  by Lemma \ref{lem-boundary Amp_l},  $K_{X_m}+\Delta_m+A_m+\frac{1}{r_l}H'_m \in  \partial \overline{\mathrm{Amp}_1}$ holds. Thus
$$r_1(K_{X_m}+\Delta_m+A_m, t_{m-1}(K_{X_m}+\Delta_m+A_m)+H'_m)=r_l-t_{m-1}
$$ holds. By Theorem \ref{thrm-finite models} and Theorem \ref{thrm-rationality} together with the fact that $X_i$  has only $\mathbb{Q}$-factorial rational singularities, the denominator of  $r_1(K_{X_m}+\Delta_m+A_m, t_{m-1}(K_{X_m}+\Delta_m+A_m)+H'_m)$ can be bounded from above by some constant $M'=M'(U, X,\Delta+A )$. Therefore, once we let $M:=M'M_{m-1}$, we finish the proof of Theorem $\ref{g-rationality thm}$.
 \end{proof}

\section{Proof of Theorem \ref{main theorem}}\label{sect-proofs}

In this section, we prove the statements (1),(2) and (3) of Theorem \ref{main theorem}. The proofs will be given at the end of the section after collecting the ingredients.

\bigskip

The following is the key result used in the proof.

\begin{proposition}\label{prop-main prop}
Let $(X,\Delta)$ be a $\Q$-factorial klt pair.
Then for any ample divisor $H$ on $X$, the set $\Sigma^l(X,\Delta+H)$ of $(K_X+\Delta+H)$-negative extremal rays of $\overline{\NE}^l(X,\Delta+H)$ is finite.
\end{proposition}
The proof is a modification of the proof of the Cone Theorem (Theorem \ref{thrm-Cone theorem}) into our general situation.

\begin{proof}
First of all, since $H$ is ample, we may assume that $\Delta$ is a $\Q$-divisor by slightly perturbing $\Delta$ to a smaller effective $\Q$-divisor $\Delta'$ and adding the difference $\Delta-\Delta'$ to $H$.
Furthermore, since we can find an ample $\Q$-divisor $H'$ such that $\Sigma^l(X,\Delta+H)\subseteq\Sigma^l(X,\Delta+H')$,
it is enough to show the finiteness of $\Sigma^l(X,\Delta+H')$. Thus we may also assume that $H$ is a $\Q$-divisor. Clearly, we may still assume that $(X,\Delta+H)$ is klt.

If $\Sigma^l(X,\Delta+H)\neq\emptyset$, then Theorem \ref{thrm-l-nef criterion} implies that $K_X+\Delta+H\not\in\overline{\Amp}_{d-l}(X)$.
Let $L$ be a Cartier divisor in $(\mathbb R_{\geq0}[K_X+\Delta+H]+\Amp_1(X))\cap\partial\overline{\Amp}_{d-l}(X)$.
Then we have
$$
F_L\not\subseteq\overline{\NE}^l(X,\Delta)_{K_X+\Delta+H\geq0}
$$
where $F_L:=\overline{\bNM}^l(X)\cap\{\eta\in N_1(X)|\;\eta\cdot L=0\}$ is a non-trivial face of $\overline{\bNM}^l(X)$.

Our goal is to perturb $L$ to a Cartier  divisor $L'$ within $(\mathbb R_{\geq0}[K_X+\Delta+H]+\Amp_1(X))\cap\partial\overline{\Amp}_{d-l}(X)$ so that
$F_{L'}\subseteq F_L$ and $\dim F_{L'}=1$.
If we already have  $\dim F_L=1$, then no perturbation is necessary. 
Thus let us assume that $\dim F_L>1$.
We can find an ample $\Q$-divisor $H'$ and an open subcone $U\subseteq\Amp_1(X)$ containing $H'$ such that there are finitely many minimal models for the klt pairs $(X,\Delta+H+H'')$ for $H''\in U$. Then by Proposition \ref{g-rationality thm}, there exists a constant $M(U,X,\Delta+H)>0$ depending only on the pair $(X,\Delta)$ and $U$ such that for any $\Z$-divisor $D\in(\R_{>0}[K_X+\Delta+H]+U)\cap\overline{\Amp}_{d-l}(X)$, the threshold $r_l(D;K_X+\Delta+H)$ is a rational number of the form $\frac{\alpha}{\beta}$ for some positive integers $\alpha,\beta$ with $0<\beta<M(U,X,\Delta+H)$.
Take an ample Cartier divisor $A\in U$ and consider the following sequence with $n\in\mathbb N$,
$$
r_L(n,A):=r_{d-l}(nL+A;K_X+\Delta+H)=\sup\{t\in\R_{\geq 0}| n L+A+t(K_X+\Delta+H)\in\overline{\Amp}_{d-l}(X)\}.
$$
Note that $nL+A\in(\R_{>0}[K_X+\Delta+H]+U)\cap\overline{\Amp}_{d-l}(X)$.
Since $L\in\overline{\Amp}_{d-l}(X)$, $r_L(n,A)$ is non-decreasing as $n$ increases.
Note that $r_L(n,A)$ is also a sequence bounded above since
$$
r_L(n,A)\leq \frac{(n L+A)\cdot z}{-(K_X+\Delta+H)\cdot z}=\frac{A\cdot z}{-(K_X+\Delta+H)\cdot z}
$$
for any fixed $z\in F_L$.
Therefore, $r_L(n,A)$ stabilizes for all sufficiently large $n>0$, i.e., there exists $n_0>0$ such that for all $n\geq n_0$,
$r_L(n,A)=r_L(n_0,A)$. Denote this constant by $r_L(A):=r_L(n_0,A)$.

Note that by construction, $L$ and $L(A):=L(n_0,A)=(n_0+1) L+A+r_L(A)(K_X+\Delta+H)$ are both in $(\mathbb R_{\geq0}[K_X+\Delta+H]+\Amp_1(X))\cap\partial\overline{\Amp}_{d-l}(X)$. Furthermore, by choosing a general $A\in U$, we have
$F_{L(A)}\subseteq F_{L}$ and $F_{L(A)}\not\subseteq \overline{\NE}^l(X)_{K_X+\Delta+H\geq0}$.

Let $\{A_i\in U|i=1,2,\cdots,\rho\}$ where $\rho:=\dim N^1(X)_{\R}$ be the set of general ample divisors in $U$ which forms a basis of $N^1(X)_{\R}$.
Then not all the hyperplanes $\{\eta|\; L(A_i)\cdot \eta =0\}$ contain $F_L$.
Thus there exists an ample divisor $A_i$ such that $0\subsetneq F_{L(A_i)}\subsetneq F_L$. Consequently, we have
$$
\dim F_{L(A_i)}<\dim F_L.
$$
By induction on the dimension, we obtain a divisor $L(A)\in (\mathbb R_{\geq0}[K_X+\Delta+H]+\Amp_1(X))\cap\partial\overline{\Amp}_{d-l}(X)$ such that
$F_{L(A)}\subseteq F_L$ and $\dim F_{L(A)}=1$.
The ray $R:=F_{L(A)}$ is an element of $\Sigma^l(X,\Delta+H)$.
Therefore, if we let $\Sigma$ be the set of all the $(K_X+\Delta+H)$-negative extremal rays of $\overline{\NE}^l(X,\Delta+H)$ that are of the form $F_{L(A)}$ for some
$L\in (\mathbb R_{\geq0}[K_X+\Delta+H]+\Amp_1(X))\cap\partial\overline{\Amp}_{d-l}(X)$ and an ample divisor $A$ in some open subcone $U\subseteq\Amp_1(X)$ (associated with $L$), then $\Sigma\subseteq\Sigma^l(X,\Delta+H)$.

We next claim the following equality:
\begin{equation}\label{*}
\overline{\NE}^l(X,\Delta+H)=\overline{\NE}(X)_{K_X+\Delta+H\geq 0}+\overline{\sum_{R\in\Sigma}R}.
\end{equation}
Since $\Sigma\subseteq\Sigma^l(X,\Delta+H)$ holds, the inclusion $\supseteq$ holds in (\ref{*}).
If this inclusion is strict, there exists an open subcone $U\subseteq\Amp_1(X)$ such that for any ample divisor $A\in U$ with $r_{d-l}:=r_{d-l}(A;K_X+\Delta+H)$,
$A+r_{d-l}(K_X+\Delta+H)$ is positive on the cone on the right hand side of (\ref{*}) except at the origin.
Let $L=A+r_{d-l}(K_X+\Delta+H)$ for an ample Cartier divisor $A\in U$. Then by arguing as above, we can find a general ample divisor $A'\in U$ such that $\dim F_{L(A')}=1$. Thus $R=F_{L(A')}$ is a ray in $\Sigma^l(X,\Delta+H)$. This is a contradiction because its supporting divisor $L(A')$ is positive on the cone $\overline{\NE}^l(X,\Delta+H)\setminus\{0\}$.

Next we remove the closure in the last summation in (\ref{*}) by showing that the rays in $\Sigma$ has no accumulation in the half space $K_X+\Delta+H<0$.
Suppose that the rays in $\Sigma$ accumulate somewhere in $K_X+\Delta+H<0$. Then it implies that by considering the dual in $N^1(X)$, for some sufficiently small open subcone $U\subseteq\Amp_1(X)$, there exists a convergent sequence of rays of the form $R=\R_{>0}[L(A)]$ for some $L\in(\R_{>0}[K_X+\Delta+H]+U)\cap\partial\overline{\Amp}_l(X)$ and $A\in U$.
Let $\eta_R$ be the generator of the ray $R$.
Choose ample Cartier divisors $A'_i$ ($i=1,\cdots,\rho-1$) in $U$ which together with $K_X+\Delta+H$ form a basis of $N^1(X)$.
Let $\mathbb P(N^1(X))$ be the projectivization of $N^1(X)_{\R}$ with respect to the basis $\{A'_i|\;i=1,2,\cdots,\rho-1\}\cup\{K_X+\Delta+H\}$.
Since the rays accumulate in the half space $K_X+\Delta+H<0$,
they will also form an accumulation point in the affine space $V$ given by $K_X+\Delta+H< 0$.
However, the affine coordinate system
$$
\eta_R\in V\mapsto \left(\frac{A'_1\cdot\eta_R}{(K_X+\Delta+H)\cdot\eta_R},\frac{A'_2\cdot\eta_R}{(K_X+\Delta+H)\cdot\eta_R},\cdots,
\frac{A'_{\rho-1}\cdot\eta_R}{(K_X+\Delta+H)\cdot\eta_R}\right)
$$
does not allow such accumulation since each $\frac{A'_i\cdot\eta_R}{(K_X+\Delta+H)\cdot\eta}$ is a rational number of the form $\frac{\alpha}{\beta}$ for some integers $\alpha,\beta$ with $0<\beta<M(X,\Delta+H,U)$ by Proposition \ref{g-rationality thm}.
Thus we obtain
$$
\overline{\NE}^l(X,\Delta+H)=\overline{\NE}(X)_{K_X+\Delta+H\geq 0}+\sum_{R\in\Sigma}R.
$$
This implies that the rays in $\Sigma^l(X,\Delta+H)$ can only accumulate on the plane $K_X+\Delta+H=0$.

The finiteness of $\Sigma^l(X,\Delta+H)$ can be seen as follows.
Suppose that there are infinitely many rays in $\Sigma^l(X,\Delta+H)$.
Then the rays in $\Sigma^l(X,\Delta+H)$ can possibly accumulate either toward the hyperplane $K_X+\Delta+H=0$ or somewhere in the half space $K_X+\Delta+H<0$. The latter case cannot occur by what we have shown above.
Thus assume that there is a sequence of rays in $\Sigma^l(X,\Delta+H)$ which accumulate toward the hyperplane $K_X+\Delta+H=0$.
Note that for any sufficiently small $\eps>0$, we have $\Sigma^l(X,\Delta+H)\subseteq\Sigma^l(X,\Delta+(1-\eps)H)$.
Therefore, if the rays of $\Sigma^l(X,\Delta+H)$ have an accumulation on $K_X+\Delta+H=0$, then it is also an accumulation of the rays in $\Sigma^l(X,\Delta+(1-\eps)H)$. Furthermore, such accumulation is in the half space $K_X+\Delta+(1-\eps)H<0$. This is again a contradiction to what we proved above.
\end{proof}

\begin{corollary}\label{cor-(1-e) finite}
Let $(X,\Delta)$ be a $\Q$-factorial klt pair and $H$ an ample divisor on $X$.
Then for any $\eps$ such that $0<\eps\leq 1$, the set $\Sigma^l(X,\Delta+\eps H)_{K_X+\Delta+H<0}$ of $(K_X+\Delta+H)$-negative extremal rays of $\overline{\NE}^l(X,\Delta+H)$ is finite.
\end{corollary}
\begin{proof}
It follows from the inclusion $\Sigma^l(X,\Delta+\eps H)\supseteq \Sigma^l(X,\Delta+\eps H)_{K_X+\Delta+H<0}$ and the fact that  $\Sigma^l(X,\Delta+\eps H)$ is finite by Proposition \ref{prop-main prop}.
\end{proof}

\bigskip

The following results explain where the rays of $\Sigma^l(X,\Delta)$ can accumulate.

\begin{proposition}\label{prop-S^l_H=sum}
Let $H$ be an ample divisor on $X$.
If $0<\eps\leq\eps'\leq 1$, then we have
$$
\Sigma^l(X,\Delta)_{K_X+\Delta+H<0}\supseteq \Sigma^l(X,\Delta+\eps H)_{K_X+\Delta+H<0}\supseteq \Sigma^l(X,\Delta+\eps' H)_{K_X+\Delta+H<0}.
$$
Furthermore, we have either
$$
\Sigma^l(X,\Delta)_{K_X+\Delta+H<0}=\bigcup_{0<\eps\leq 1}\Sigma^l(X,\Delta+\eps H)_{K_X+\Delta+H<0}
$$
or
$$
\Sigma^l(X,\Delta)_{K_X+\Delta+H<0}=\overline{\bigcup_{0<\eps\leq 1}\Sigma^l(X,\Delta+\eps H)_{K_X+\Delta+H<0}}
$$
where the closure is defined by the following topology:  Take some  plane $P$ such that  $P\cap \overline{NE}^l(X,\Delta)$ is compact and   for rays $R, R'$, $||R,R'||:=||R\cap P, R'\cap P||$ gives the induced topology on the set of extremal rays. Note that the topology is independent of the choice of $P$.
\end{proposition}
\begin{proof}
Note first the following:
$$
\overline{\NE}^l(X,\Delta)\subseteq \overline{\NE}^l(X,\Delta+\eps H)\subseteq \overline{\NE}^l(X,\Delta+\eps' H).
$$
By collecting the $(K_X+\Delta+H)$-negative extremal rays of these cones, we obtain the first inclusions.
(The smaller the $\eps>0$ is, the larger part of the cone $\overline{\bNM}^l(X)$ is exposed to the boundary of the cone $\overline{\NE}^l(X,\Delta+\eps H)$.)
This immediately shows that
\begin{equation}\label{Sigma_inclusion}
\Sigma^l(X,\Delta)_{K_X+\Delta+H<0}\supseteq\bigcup_{0<\eps\leq 1}\Sigma^l(X,\Delta+\eps H)_{K_X+\Delta+H<0}.
\end{equation}

Now suppose that the inclusion is strict and let $R$ be a ray in $\Sigma^l(X,\Delta)_{K_X+\Delta+H<0}$ which is not contained in $\bigcup_{0<\eps\leq 1}\Sigma^l(X,\Delta+\eps H)_{K_X+\Delta+H<0}$. Then since $R$ is a $(K_X+\Delta+H)$-negative extremal ray of $\overline{\NE}^l(X,\Delta)$, there exists a hyperplane $F$ defined by the vanishing of a class $\eta\in N^1(X)$ such that $R=F\cap\overline{\NE}^l(X,\Delta)$ and $\eta$ is non-negative on $\overline{\NE}(X)_{K_X+\Delta\geq 0}\setminus\{0\}$.

If $\eta$ is positive on $\overline{\NE}(X)_{K_X+\Delta\geq 0}\setminus\{0\}$, $\eta$ is also positive on $\overline{\NE}(X)_{K_X+\Delta+\eps' H\geq 0}\setminus\{0\}$ for some sufficiently small $\eps'>0$. Since $F=\{\eta=0\}$ is still a supporting plane of $\overline{\NE}^l(X,\Delta+\eps' H)$ for the ray $R$, $R$ is also a $(K_X+\Delta+H)$-negative extremal ray of $\overline{\NE}^l(X,\Delta+\eps' H)$ for some sufficiently small $\eps'>0$, that is, $R\in\Sigma^l(X,\Delta+\eps' H)_{K_X+\Delta+H<0}$.
Thus we have a contradiction and the hyperplane $F=\{\eta=0\}$ must support both $\overline{\NE}(X)_{K_X+\Delta\geq 0}$ and $\overline{\bNM}^l(X)$ simultaneously.

By assumption, for any hyperplane $F(\eps)$ supporting both $\overline{\NE}(X)_{K_X+\Delta+\eps H\geq 0}$ and $\overline{\bNM}^l(X)$, we have $R\not\subseteq F(\eps)\cap\overline{\bNM}^l(X)$. However, by taking $\eps>0$ arbitrarily small, we can find a  $(K_X+\Delta+\eps H)$-negative extremal ray $R'$ of the face $F(\eps)\cap\overline{\bNM}^l(X)$ which is sufficiently close to the ray $R$.
Otherwise, there would be an open cone $U'$ such that $R$ is the only $(K_X+\Delta+H)$-negative extremal ray of $\overline{\NE}^l(X,\Delta)$ contained in $U'$. 
This implies that $\overline{\NE}^l(X,\Delta)$ is rational polyhedral locally at $R$. Therefore, we can find a hyperplane supporting $\overline{\NE}^l(X,\Delta)$ at $R$ which intersects the half cone $\overline{\NE}(X)_{K_X+\Delta\geq0}$ only at the origin. This is a contradiction to what we proved above about the supporting hyperplane $F=\{\eta=0\}$ of $R$.

Hence, $R$ is a limit of the rays in $\cup_{0<\eps\leq 1}\Sigma^l(X,\Delta+\eps H)_{K_X+\Delta+H<0}$ and we are done.
\end{proof}

\bigskip

\begin{proof}[Proof of (1) of Theorem \ref{main theorem}]
First of all, we have
$$
\overline{\NE}^l(X,\Delta+\eps H)=\overline{\NE}^l(X,\Delta+\eps H)_{K_X+\Delta+ H\geq 0}+\sum\limits_{R\in\Sigma^l(X,\Delta+\eps H)_{K_X+\Delta+H<0}}R
$$
where $\Sigma^l(X,\Delta+\eps H)_{K_X+\Delta+H<0}$ is finite by Corollary \ref{cor-(1-e) finite}.
Thus we have
$$
\begin{array}{rl}
 \bigcap\limits_{0<\eps\leq 1}\overline{\NE}^l(X,\Delta+\eps H) &=\bigcap\limits_{0<\eps\leq 1}\left(\overline{\NE}^l(X,\Delta+\eps H)_{K_X+\Delta+ H\geq 0}+\sum\limits_{R\in\Sigma^l(X,\Delta+\eps H)_{K_X+\Delta+H<0}}R\right)\\
   & =\overline{\NE}^l(X,\Delta)_{K_X+\Delta+H\geq 0}+\overline{\sum R}
\end{array}
$$
where the last summation under the closure is taken over the rays $R$ in the union $\bigcup\limits_{0<\eps\leq 1}\Sigma^l(X,\Delta+\eps H)_{K_X+\Delta+H<0}$.
However, Proposition \ref{prop-S^l_H=sum} implies that we can take the summation over the rays in $\Sigma^l(X,\Delta)_{K_X+\Delta+H<0}$ and remove the closure.
Furthermore, since
$$\overline{\NE}^l(X,\Delta)=\bigcap\limits_{0<\eps\leq 1}\overline{\NE}^l(X,\Delta+\eps H),$$ we have
$$
\overline{\NE}^l(X,\Delta)=\overline{\NE}^l(X,\Delta)_{K_X+\Delta+H\geq0}+\sum_{R\in\Sigma^l(X,\Delta)_{K_X+\Delta+H<0}} R.$$
Finally since the ample divisor $H$ can be chosen arbitrarily small, we obtain
$$
\overline{\NE}^l(X,\Delta)=\overline{\NE}^l(X,\Delta)_{K_X+\Delta\geq0}+\sum_{R\in\Sigma^l(X,\Delta)} R.$$
\end{proof}
\medskip
\begin{proof}[Proof of (2) of Theorem \ref{main theorem}]
Since $\Sigma^l(X,\Delta+H)$ is finite by Proposition \ref{prop-S^l_H=sum},
we have
$$
\overline{\NE}^l(X,\Delta+H)=\overline{\NE}^l(X,\Delta+H)_{K_X+\Delta+H\geq0}+
\sum_{R\in\Sigma^l(X,\Delta+H)}R.
$$
\end{proof}

Note that the rays in $\Sigma^l(X,\Delta)_{K_X+\Delta+H<0}$ can possibly accumulate while the rays in $\Sigma^l(X,\Delta+H)$ are discrete since it is finite.

\begin{proof}[Proof of (3) of Theorem \ref{main theorem}]
Let $H$ be an ample divisor. Assume that $\Sigma^l(X,\Delta+H)\neq\emptyset$ and let $R\in\Sigma^l(X,\Delta+H)$.
If $l=0$, then it follows from \cite{Araujo}.
Thus we assume below that $l>0$.
By the proof of Proposition \ref{prop-main prop}, $R$ is of the form $R=F_D:=\overline{\bNM}^l(X)\cap\{\eta\in N_1(X)|\eta\cdot D=0\}$
for some $\Q$-divisor $D\in (\R_{\geq 0}[K_X+\Delta+H]+\Amp_1(X))\cap\partial\overline{\Amp}_{d-l}(X)$.
(See the proof of Proposition \ref{prop-main prop} for the precise description of $D$.)
Thus we can find an ample $\Q$-divisor $G$ such that $D=K_X+\Delta+H+G$ and we may assume that $(X,\Delta+H+G)$ is klt.

By (\ref{expanded Mori cone}), there exists a small $\Q$-factorial modification $\varphi:X\dashrightarrow X'$ such that $R$ is an extremal ray of $\overline{\NM}^l(X,X')$. This map $\varphi$ can be obtained by running the MMP on $(X,\Delta+H+G)$ with scaling of $aG$ for $a>0$ where we may assume that $(X,\Delta+H+(1+a)G)$ is klt and $K_X+\Delta+H+(a+1)G$ is ample.
By \cite{bchm}, this MMP terminates and we have a nef divisor $D':=\varphi_*(K_X+\Delta+H+G)=K_{X'}+\Delta'+H'+G'$.
We have $R=\overline{\NM}^l(X,X')\cap\varphi^*\{\eta\in N_1(X')|\eta\cdot D'=0\}$.
If we consider $R$ as a ray in $N_1(X')$, then since $D'$ is nef, $R$ is also an extremal ray of the Mori cone $\overline{\NE}(X')$. Furthermore, it is $(K_{X'}+\Delta'+H')$-negative. Therefore, $R$ as a $(K_{X'}+\Delta'+H')$-negative extremal ray of $\overline{\NE}(X')$ in $N_1(X')$, it is spanned by a rational curve $C'$ on $X'$.
Thus the ray $R$ as a ray in $N_1(X)$ is spanned by the numerical pull back of the rational curve $C'$
\end{proof}

We note that the statement (3) does not say anything about the limit of the rays in $\Sigma^l(X,\Delta)$ contained in the hyperplane supporting both cones $\overline{\NE}(X)_{K_X+\Delta\geq 0}$ and $\overline{\bNM}^l(X)$. This is analogous to the case of the Cone Theorem.  The limit of the $(K_X+\Delta)$-negative extremal rays of $\overline{\NE}(X)$ is $(K_X+\Delta)$-trivial and is not guaranteed to be spanned by a rational curve.

\begin{remark}\label{rmk-Batyrev}
In \cite{Batyrev}, Batyrev proved that for a $\Q$-factorial terminal threefold $X$ (with $\Delta=0$) and any ample divisor $A$ on $X$, there exists a finite subset $\Sigma'\subseteq\Sigma^0(X,0)$ of $(K_X+A)$-negative extremal rays such that
$$
\overline{\NE}^0(X,0)=\overline{\NE}^0(X,0)_{K_X+A\geq 0}+\sum_{R\in\Sigma'}R.
$$
This implies that as in the Cone Theorem, the extremal rays in $\Sigma^0(X,0)$ can only accumulate toward the plane $K_X=0$.
Note that this result is slightly stronger than Theorem \ref{thrm-batyrev} (or our result Theorem \ref{main theorem} for the case $l=0$) which allows an accumulation of the $K_X$-negative extremal rays of $\overline{\NE}^0(X,0)$ away from the hyperplane $K_X=0$.
Batyrev's proof is based on the boundedness of terminal Fano threefolds. As is well known, a far more general boundedness result (known as Borisov-Alexeev-Borisov Conjecture) has been proved recently by Birkar \cite{B1},\cite{B2} and this result can be applied to obtain a similar statement for $\Q$-factorial klt pairs. See also \cite[Section 6.]{Lehmann-cone}. It is interesting to ask whether the same holds for the rays in $\Sigma^l(X,\Delta)$ for a $\Q$-factorial klt pair $(X,\Delta)$ and for all $l=0,1,2,\cdots,\dim X-1$: for any ample divisor $A$ on $X$,
$$
\overline{\NE}^l(X,\Delta)=\overline{\NE}^l(X,\Delta)_{K_X+\Delta+A\geq 0}+\sum_{R\in\Sigma''}R
$$
where $\Sigma''$ is a finite subset of $\Sigma^l(X,\Delta)$. Note that since the case where $l=\dim X-1$ corresponds to the Cone Theorem, we are left to verify the cases $l=1,2,\cdots,\dim X-2$.
\end{remark}

\bigskip


\end{document}